%% file: document.tex
\documentclass[a4paper, 11pt,reqno]{article}
\usepackage[top=100pt,bottom=60pt,left=96pt,right=92pt]{geometry}

\input{instellingen}
\usepackage[english]{babel}
\title{{A Hermitian refinement of symplectic Clifford analysis}\label{hermC}}
\author{David Eelbode\footnote{UAnwerpen, Middelheimlaan 1, 2020 Antwerp  Belgium \url{david.eelbode@uantwerpen.be}
	} \qquad  Guner Muarem\footnote{UAnwerpen, Middelheimlaan 1, 2020 Antwerp  Belgium	\url{guner.muarem@uantwerpen.be}}}

\begin{document}

\maketitle

\begin{abstract}
In this paper we develop the Hermitian refinement of symplectic Clifford analysis, by introducing a complex structure $\mathbb{J}$ on the canonical symplectic manifold $(\R^{2n},\omega_0)$. This gives rise to two symplectic Dirac operators $D_s$ and $D_t$ (in the sense of Habermann \cite{MR2252919}), leading to a $\mathfrak{u}(n)$-invariant system of equations on $\R^{2n}$. We discuss the solution space for this system, culminating in a Fischer decomposition for the space of polynomials on $\R^{2n}$ with values in the symplectic spinors. To make this decomposition explicit, we will construct the associated embedding factors using a transvector algebra.
\end{abstract}

 \section{Introduction} 

\noindent This paper is to be situated in the framework of so-called symplectic (Clifford) analysis, a function theory which focuses on solutions for the symplectic Dirac operator $D_s$ (see \cite{MR2252919} for basic information). This is a first-order differential operator which generalises the classical Dirac operator\footnote{In itself already a generalisation of the operator introduced by P.A.M. Dirac to describe the behaviour of electrons.} on $\R^n$ in the following sense: instead of the typical spin group invariance characterising the standard Dirac operator, which is {\em orthogonal} in nature because $\mathsf{Spin}(n)$ is a double cover of the orthogonal group $\mathsf{SO}(n)$, one will get {\em symplectic} invariance. It is still a Dirac operator though, because the symplectic Dirac operator $D_s$ acts on functions taking values in the symplectic generalisation of the spinor space $\mathbb{S}$ from the orthogonal setting, and because its solutions can be related to solutions for the Laplace operator (this point will also be treated in this paper, although we immediately note that it is more subtle than the orthogonal case, where one simply has that the square of the classical Dirac operator gives the Laplace operator). In this paper, we will consider a symmetry reduction from the symplectic Lie group $\mathsf{Sp}(2n)$ to its subgroup $\mathsf{U}(n)$. To do this, we will introduce a complex structure $\mathbb{J}$ on the symplectic manifold $(\R^{2n},\omega_0)$. Just like in the orthogonal case, introducing a complex structure $\mathbb{J}$ on $\R^{2n}$ will then lead to a refinement of symplectic Clifford analysis. 
\par
Before introducing this symplectic analogue, we will shortly review the basics of the Hermitian refinement of (classical) Clifford analysis, breaking down the symmetry of the orthogonal group $\mathsf{SO}(2n)$ to the unitary group $\mathsf{U}(n)$. This leads to a system of two Hermitian Dirac operators which are invariant under the unitary group $\mathsf{U}(n)$ (more precisely: its double cover, realised as a subgroup of the spin group). It is well-known that the algebraic structure generated by these Hermitian Dirac operators and their (Fischer) dual operators gives rise to the Lie superalgebra $\mathfrak{osp}(2|2)$, which then leads to a Howe dual pair $\mathsf{U}(n)\times \mathfrak{osp}(2|2)$ with an associated Fischer decomposition for the space of spinor-valued complex harmonics (see \cite{BDES}, or Theorem \ref{HFD}).
\par
The symplectic version of this last theorem will be the main result of this paper. To do so we will define the operator $D_t$, related to the symplectic Dirac operator $D_s$ under a suitable complex structure. Next we describe the unitary invariance of this pair of operators $(D_s,D_t)$ and determine the algebra generated by them, leading to the dual pair $\mathsf{U}(n) \times \mathfrak{su}(1,2)$. Finally, we will use this dual pair to obtain a Fischer decomposition for the space of spinor-valued bi-harmonics on the space $\R^{2n} \cong \C^n$. 
\section{Hermitian Clifford analysis}\label{Hermitian}
\noindent 
The key ingredient in the transition from the classical case to Hermitean refinement is a complex structure $\mathbb{J}$, given by
\[ \mathbb{J}=\begin{pmatrix}
0& I_n\\ -I_n&0
\end{pmatrix} \in \mathsf{SO}(2n)\ . \]
On the standard orthonormal basis $\{e_1,\ldots,e_{2n}\}$ for $\R^{2n}$ this complex structure $\mathbb{J}$ thus acts as follows: $\mathbb{J}(e_k)=-e_{k+n}$ and $\mathbb{J}(e_{k+n})=e_{k}$, for all $1 \leq k \leq n$. Note that the coordinates will be referred to as $(x_1,\ldots,x_n,y_1\ldots,y_n)\in\mathbb{R}^{2n}$ in this paper. The (classical) Dirac operator $\up_x$ and its (Fischer) dual $\ux$ on $\R^{2n}$ are given by
\begin{align*}
\ux = \sum_{k=1}^n \big( e_kx_k+e_{n+k}y_k \big) \quad \mbox{and} \quad
\up_x =\sum_{k=1}^n \big( e_k\partial_{x_k}+e_{n+k}\partial_{y_k} \big)\ ,
\end{align*}
and using the complex structure $\mathbb{J}$ one can define a twisted (or rotated) version:	
\begin{align*}
\ux^{\mathbb{J}}:&=\mathbb{J}[\ux]=\sum_{k=1}^n(y_ke_{2k-1}-x_ke_{2k})\\
{\partial^{\mathbb{J}}_{\ux}}:&=\mathbb{J}[\up_x]=\sum_{k=1}^n(e_{2k-1}\partial_{y_k} - e_{2k}\partial_{x_k})\ .
\end{align*}
Note that $\mathbb{J}$ acts on the basis vectors $e_k$ here. In the definition below, $\mathbb{S}$ stands for the space of Dirac spinors on $\R^{2n}$ (i.e. the sum of positive and negative Weyl spinors), typically realised using a minimal left ideal in the complex Clifford algebra $\mathbb{C}_{2n}$. 
\begin{definition}
	An $\mathbb{S}$-valued function $f$ on $\R^{2n}$ satisfying the equations $\up_x f = {\partial^{\mathbb{J}}_{\ux}} f=0$ is called {\em Hermitian monogenic} (or $h$-monogenic). One further defines the space 
	\begin{align*}
	\mathcal{M}_k^h(\R^{2n},\mathbb{S}):=\mathcal{P}_k(\R^{2n},\mathbb{S})\cap \big(\ker(\up_x)\cap\ker(\partial^{\mathbb{J}}_{\ux})\big)\ , 
	\end{align*}
	containing $k$-homogeneous polynomials which are $h$-monogenic.
\end{definition}
There is another way to introduce a pair of Hermitian Dirac operators, using the Witt basis for $\R^{2n}\cong \C^n$:
\begin{definition}
	The Witt basis vectors for $\C^n$ are defined as 
	\begin{align*}
	\mathfrak{f}_k = \frac{1}{2}\left(e_k - i e_{n+k}\right) \quad \mbox{and} \quad \mathfrak{f}_k^{\dagger} = -\frac{1}{2}\left(e_k + i e_{n+k}\right)\ .
	\end{align*}
\end{definition}
\noindent
The Hermitian Dirac operators and their adjoints are then introduced as follows:
\begin{align}
(\up_z,\up_z^\dagger) = \left(\sum_{k=1}^n\mathfrak{f}_k^{\dagger}\partial_{z_k}, \sum_{k=1}^n\mathfrak{f}_k\partial_{\overline{z}_k}\right) \quad \mbox{and} \quad (\uz,\uz^\dagger) = \left(\sum_{k=1}^n\mathfrak{f}_k z_k,\sum_{k=1}^n\mathfrak{f}_k^{\dagger}\overline{z}_k\right)\ .
\end{align}\label{dz}
\noindent 
The Dirac operator on $\R^{2n}$ can then be written as $\up_x = 2(\up_z^{\dagger}-\up_z)$, and similarly for its adjoint $\ux = \uz - \uz^\dagger$. This shows why Hermitian Clifford analysis is a {\em refinement} of classical Clifford analysis: $h$-monogenics always belong to the kernel of the Dirac operator $\up_x$, but not necessarily the other way around. 
\begin{remark}
	The operators $\up_z$ and $\up_z^{\dagger}$ can be used to factorise the Laplacian $\Delta$ on $\R^{2n}$, since $\Delta = 4(\up_z^{\dagger}\up_z+\up_z\up_z^{\dagger}) = 4\{\up_z^{\dagger},\up_z\}$. This implies that every $h$-monogenic function is harmonic.
\end{remark}
The Witt basis can be used to define a simple model for the spinor space $\mathbb{S}$ in the orthogonal setting. Indeed, one has that $\mathbb{S} = \C_{2n}I$ where $I$ is the primitive idempotent $I=I_1\cdots I_n$ with $I_k=\mathfrak{f}_k\mathfrak{f}_k^{\dagger}$ for all $1 \leq k \leq n$. Due to the fact that $\mathfrak{f}_kI=0$, the spinor space $\mathbb{S}$ reduces to the action of the (complex) Grassmann algebra $\Lambda^{\dagger}_n$ spanned by the elements $\{\mathfrak{f}^{\dagger}_k\}_{k=1}^n$, which means that we can write $\mathbb{S} = \Lambda^{\dagger}_nI$. This implies that the spinor space decomposes into `homogeneous parts' as
$$
\mathbb{S} = \bigoplus_{r=0}^n\mathbb{S}_{(r)}=\bigoplus_{r=0}^n (\Lambda^{\dagger}_n)^{r}I$$ where $(\Lambda^{\dagger}_n)^{r}$ consists of all $r$-vectors in $\Lambda^{\dagger}_n$. This decomposition is certainly not spin-invariant, but it does carry a (left) multiplicative $\mathfrak{u}(n)$-action, see e.g. \cite{BDES}. The invariant subspaces $\mathbb{S}_{(r)}$ can be characterised as eigenspaces of the so-called {\em spin-Euler operator} $\beta = n - \sum_k I_k$, measuring the `degree of homogeneity' $r$. 
\par
Now that we have introduced the Hermitian Dirac operators $\up_z$ and $\up_z^{\dagger}$ and discussed the decomposition of the spinor space $\mathbb{S}$ into $\mathfrak{u}(n)$-irreducible representations, we focus on the algebra generated by these Dirac operators and their (Fischer) duals. Recall that in the orthogonal case the operators $\up_x$ and $\ux$ give rise to the Lie superalgebra $\mathfrak{osp}(1|2)$, whereas in the Hermitian setting one obtains the algebra $\mathfrak{osp}(2|2)$. For further details, such as the definitions for the generators and their (anti-)commutation relations, we refer to \cite{BDES}. There, one can also find the following theorem:
\begin{theorem}[Hermitian Howe duality]\label{hermitianhowe}
	The space $\mathcal{P}(\R^{2n},\mathbb{S})$ of spinor-valued polynomials on $\R^{2n}$ can be decomposed in a multiplicity-free way, using the Howe dual pair $\U(n) \times \mathfrak{osp}(2|2)$. 
\end{theorem}
\noindent 
Using the identification $(z_1,\dots,z_n)\in \C^n\leftrightarrow(x_1,\dots,x_n,y_1,\dots,y_n)\in\R^{2n}$, where we have put $z_k = x_k + iy_k$, we get a canonical identification 
\begin{align*}
P(\ux,\uy)\in\mathcal{P}(\R^{2n},\C)\ \leftrightarrow\ P(\uz,\overline{\uz})\in \mathcal{P}(\C^n,\C)\ .
\end{align*}
Further structure can then be imposed on this space of polynomials: 
\begin{definition}
	Let $a,b$ be two non-negative integers. A polynomial $P(\uz,\overline{\uz})$ in $\mathcal{P}(\C^n,\C)$ is said to be \textit{bi-homogeneous} of degree $(a,b)$ if
	\begin{align*}
	P(\lambda\uz,\overline{\lambda}\overline{\uz}) = \lambda^a\overline{\lambda}^bP(\uz,\overline{\uz})    
	\end{align*}
	for all $\uz\in\C^n$ and $\lambda\in \C$. We denote the space of the \textit{bi-homogeneous} polynomials of degree $(a,b)$ by $\mathcal{P}_{a,b}(\C^n,\C)$ or $\mathcal{P}_{a,b}(\R^{2n},\C)$.
\end{definition}
\begin{remark}
	Recall that in the classical case, the degree of homogeneity of a polynomial $P_k\in\mathcal{P}_k(\R^{n},\C)$ was measured by the Euler operator $\mathbb{E}=\sum_{j=1}^n x_j\partial_{x_j}$. In the Hermitian setting, one defines the so-called \textit{$h$-Euler operators} as 
	\begin{align*}
	\mE_z=\sum_{j=1}^nz_j\partial_{z_j}\quad \text{ and } \quad \mE_z^\dagger =\sum_{j=1}^n\overline{z}_j\partial_{\overline{z}_j}\ ,
	\end{align*}
	the eigenvalues of which are precisely the bi-degree parameters $a$ and $b$. 
\end{remark}
\begin{definition}
	We define $\mathcal{H}_{a,b}(\mathbb{R}^{2n},\C)$ as the space of \textit{harmonic polynomials on $\mathbb{R}^{2n}$ of bi-degree} $(a,b)$. This means that $\mathcal{H}_{a,b}(\mathbb{R}^{2n},\C) := \mathcal{P}_{a,b}(\mathbb{R}^{2n},\C)\cap \ker(\Delta)$.
\end{definition}
\noindent
Denoting the space of $k$-harmonic polynomials on $\R^{2n}$ by means of $\mathcal{H}_k(\mathbb{R}^{2n},\C)$, we then have that 
\begin{align*}
\mathcal{H}_k(\mathbb{R}^{2n},\mathbb{C}) = \bigoplus_{k=a+b}\mathcal{H}_{a,b}(\mathbb{C}^n)\ ,
\end{align*}
which expresses a irreducible representation for $\mathsf{SO}(2n)$ at the left-hand side (harmonic polynomials of degree $k$) as a direct sum of irreducible representations for the subgroup $\mathsf{U}(n)$ at the right-hand side. A key result in understanding the Hermitean refinement of the Fischer decomposition is the following theorem:
\begin{theorem}\label{HFD}
	Suppose $H_{a,b}^{(r)}(\uz,\overline{\uz}) \in \mathcal{H}_{a,b}(\C^{n},\C)\otimes \mathbb{S}_{(r)}$ with  $0<r<n$ and $a,b > 0$, then $H_{a,b}^{(r)}(\uz,\overline{\uz})$ can be decomposed as
	\begin{align*}
	H_{a,b}^{(r)}(\uz,\overline{\uz}) &= M_{a,b}^{(r)}(\uz,\overline{\uz}) + \uz M_{a-1,b}^{(r+1)}(\uz,\overline{\uz}) + \uz^{\dagger}M_{a,b-1}^{(r-1)}(\uz,\overline{\uz}) \\ &+\left((a+r-1)\uz\uz^{\dagger}-(b+m-r-1)\uz^{\dagger}\uz\right)M_{a-1,b-1}^{(r)}(\uz,\overline{\uz}), 
	\end{align*}
	where all $M_{a,b}^{(r)}(\uz,\overline{\uz}) \in \ker(\up_z)\cap\ker(\up^{\dagger}_z)$, i.e. are $h$-monogenic. 
\end{theorem}
\begin{remark}
	There are a few trivial cases for the decomposition above, which explains the restrictions on the parameters $a, b$ and $r$ (see \cite{BDES} for more details). 
\end{remark}


\section{Hermitian symplectic Dirac operators}\label{hsdsec}
Just like in the previous section, we start by introducing the rotated (or `twisted') symplectic Dirac operator and its dual, using the complex structure $\mathbb{J}$ from above. For that purpose, we recall their definitions (see \cite{MR2252919}): 
\begin{align*}
D_s = \sum_{j = 1}^n \big(iq_j \partial_{y_j} - \partial_{q_j}\partial_{x_j}\big) \quad \mbox{and} \quad X_s = \sum_{j = 1}^n \big(y_j \partial_{q_j} + iq_jx_j\big)\ .
\end{align*}
These operators act on symplectic spinor-valued functions $f(x_1,\ldots,y_n)$, whereby the symplectic spinor space can be realised as the Schwartz space $\mathcal{S}(\R^n,\C)$ of rapidly decreasing functions on $\R^n$. These then correspond to the smooth vectors of the oscillator representation for the metaplectic group $\mathsf{Mp}(2n)$, see for instance \cite{MR2252919, kos}. We will use the variable $\underline{q} = (q_1,\ldots,q_n) \in \R^n$ as the argument for these functions, and use the symbol $\mS^{\infty}$ to refer to this spinor space (where the infinity symbol stresses the fact that the underlying metaplectic representation is infinite-dimensional). 
\begin{remark}
	We denote the even and odd part of $\mathbb{S}^\infty$, containing even and odd functions respectively, by means of $\mS^{\infty}_+$ and $\mS^{\infty}_-$. These spaces then define the (positive and negative) spinor representations for the symplectic group $\mathsf{Sp}(2n)$, or its associated Lie algebra. The highest weights of these representations are given by 
	\begin{align*}
	\mS^{\infty}_+ \leftrightarrow \left(-\frac{1}{2},\dots,-\frac{1}{2}\right)_{\mathfrak{sp}(2n)}\quad\text{ and }\quad \mS^{\infty}_- \leftrightarrow \left(-\frac{1}{2},\dots,-\frac{1}{2},-\frac{3}{2}\right)_{\mathfrak{sp}(2n)}\ ,
	\end{align*}
	and they provide examples of so-called \textit{completely pointed modules}, see \cite{BL1, BL2}.
\end{remark}
\par
Now that we have defined the values of the functions that will be considered in this paper, we can use the complex structure $\mathbb{J}$, acting on the column vector $(q_j,-i\partial_{q_j})^T$, to arrive at the following operators: 
\begin{definition}
	The twisted symplectic Dirac operator $D_t$ and its (Fischer) dual $X_t$ are given by: 
	\begin{align*}
	D_t := \mathbb{J}[D_s] = \sum_{j=1}^n \big(iq_j\partial_{x_j} + \partial_{y_j}\partial_{q_j}\big) \quad \mbox{and} \quad X_t := \mathbb{J}[X_s]  &=\sum_{j=1}^n \big(x_j\partial_{q_j} - iy_jq_j\big)\ .
	\end{align*}
\end{definition}
\begin{remark}
	The operator $D_t$ was often ignored up to this point in the context of symplectic Clifford analysis for $D_s$, mostly because the former operator does not have the $\mathfrak{sp}(2n)$-invariance properties of the latter. The operator $D_t$ is indeed {\em not} invariant under the realisation for $\mathfrak{sp}(2n)$ which appears in the setting of $D_s$ (see the first set of operators in the lemma below). However, when switching to {\em another} realisation for this symplectic Lie algebra (the second set of operators in the lemma below), $D_t$ {\em will} be invariant. As is to be expected, $D_s$ will in turn not be invariant under this new realisation. We would thus argue that both operators are symplectically invariant, just like Habermann mentions in \cite{MR2252919}, under the caveat that one has to specify for which realisation. Recall that this invariance amounts to saying that the generators below commute with the symplectic Dirac operators, so e.g. $[D_s,X_{jk}] = [D_t,\tilde{Y}_{jk}] = 0$. 
\end{remark}
\begin{lemma}\label{unitary_Lie_operators}
	We have the following two realisations of the Lie algebra $\mathfrak{sp}(2n)$ on the space $\mathcal{P}(\R^{2n},\C)\otimes \mathbb{S}^{\infty}$ of symplectic spinor-valued polynomials
	\begin{align*}
	\begin{cases}
	X_{jk}=x_j\partial_{x_k}-y_k\partial_{y_j} - (q_k\partial_{q_j}+\frac{1}{2}\delta_{jk})
	& 1\leq j\leq k\leq n
	\\Y_{jk}=x_j\partial_{y_k}+x_k\partial_{y_j} +i\partial_{q_j}\partial_{q_k}
	& j<k=1,\dots,n
	\\ Z_{jk}=y_j\partial_{x_k}+y_k\partial_{x_j} + i q_jq_k
	& j<k=1,\dots,n
	\\Y_{jj}=x_j \partial_{y_j} +\frac{i}{2}\partial_{q_j}^2
	& j=1,\dots,n
	\\Z_{jj}=y_j\partial_{x_j}+\frac{i}{2} q_j^2
	&  j=1,\dots,n
	\end{cases}
	\end{align*} 
	and 
	\begin{align*}
	\begin{cases}
	\tilde{X}_{jk}=x_j\partial_{x_k}-y_k\partial_{y_j} + q_k\partial_{q_j}+\frac{1}{2}\delta_{jk}
	& 1\leq j\leq k\leq n
	\\\tilde{Y}_{jk}=x_j\partial_{y_k}+x_k\partial_{y_j} -iq_jq_k
	& j<k=1,\dots,n
	\\ \tilde{Z}_{jk}=y_j\partial_{x_k}+y_k\partial_{x_j} -i \partial_{q_j}\partial_{q_k}
	& j<k=1,\dots,n
	\\\tilde{Y}_{jj}=x_j \partial_{y_j} -\frac{i}{2}q_j^2
	& j=1,\dots,n
	\\\tilde{Z}_{jj}=y_j\partial_{x_j}-\frac{i}{2} \partial_{q_j}^2& j=1,\dots,n
	\end{cases}
	\end{align*}
	The symplectic Dirac operator $D_s$ is invariant under the first realisation, the operator $D_t$ under the second.
\end{lemma}
\begin{proof}
	This follows from straightforward commutation relations.
\end{proof}
\begin{remark}
	The capital letters used to denote these operators are referring to the letters used in a typical matrix realisation (see e.g.\ \cite{fulton2013representation}), so that the Lie algebra homomorphism becomes immediately clear. 
\end{remark}
We will now perform a symmetry reduction, such that the operators $D_s$ and $D_t$ both become invariant under \textit{one and the same} group (or associated Lie algebra). To see that this must be the unitary group, we will first switch to matrices and characterise all symplectic matrices which commute with the complex structure $\mathbb{J}$. 
\begin{lemma} 
	We have that $\mathsf{Sp}_{\mathbb{J}}(2n):= \{M\in\mathsf{Sp}(2n)\mid M\mathbb{J} = \mathbb{J}M\}$ 
	defines a realisation for the unitary Lie group $\mathsf{U}(n)$.
\end{lemma}
\begin{proof}
	In order to see this, assume that $M$ is of the block-form $\left(\begin{smallmatrix}
	A	&	B\\
	C	&	D
	\end{smallmatrix}\right)$, where $A,B,C$ and $D$ are $(n\times n)-$matrices. There are two conditions to be met here: $M^T \mathbb{J} M = \mathbb{J}$ (being symplectic) and $M\mathbb{J} = \mathbb{J}M$. The latter tells us that $A=D$ and $B=-C$, and plugging the second condition in the first immediately yields that $M^T M = \textup{Id}$ gives the identity matrix. This is precisely the condition required to be a unitary matrix, with the well-known group isomorphism $\Phi: \mathsf{Sp}_{\mathbb{J}}(2n,\mathbb{R})\to \mathsf{U}(n):M\mapsto A+iB$.
\end{proof}
On the level of operators acting on functions, we get the following: 
\begin{lemma}\label{u-realisation}
	The Lie algebra $\mathfrak{u}(n)$ can be realised in terms of the differential operators acting on the space $\mathcal{P}(\R^{2n},\C)\otimes\mathbb{S}^{\infty}$:
	\begin{align*}
	\begin{cases}
	A_{jk} = y_j\partial_{x_k} -x_k\partial_{y_j} + y_k\partial_{x_j} - x_j\partial_{y_k}+ i (q_jq_k -\partial_{q_j}\partial_{q_k})&\quad 1\leq j<k\leq n
	\\B_{jj}=y_j\partial_{x_j}-x_j\partial_{y_j}+\frac{i}{2}\left(q_j^2-\partial_{q_j}^2\right)
	&\quad 1\leq j\leq n
	\\ C_{jk}=x_j\partial_{x_k} - x_k\partial_{x_j}+ y_j\partial_{y_k} - y_k\partial_{y_j} + q_j\partial_{q_k} - q_k\partial_{q_j}
	&\quad 1\leq j<k\leq n
	\end{cases}
	\end{align*}
\end{lemma}
\begin{proof}
	One can use the previous lemma to obtain the matrix realisation, and use lemma \ref{unitary_Lie_operators} to obtain the corresponding differential operators. Alternatively, one can also employ the fact that the operator $\sum_j B_{jj}$ corresponds to the symplectic structure $\mathbb{J}$ from the previous lemma, and use this to obtain the required linear combinations which commute with $\sum_j B_{jj}$. 
\end{proof}
\begin{corollary}
	The operators $D_s$ and $D_t$ are invariant differential operators with respect to the realisation for $\mathfrak{u}(n)$ from above.
\end{corollary}
\begin{proof}
	It is straightforward to check that $D_s$ and $D_t$ commute with all the operators from Lemma \ref{u-realisation}.
\end{proof}
This means that we can now refine the $\mathfrak{sp}(2n)$-invariant equation $D_sf=0$, by introducing a system of two $\mathfrak{u}(n)$-invariant equations: $D_sf = $ and $D_t f = 0$, for an $\mathbb{S}^\infty$-valued function $f$ (note that we will mostly stick to polynomials in this paper). On the analogy with the classical case, we have the following:
\begin{definition} A function which is a solution for both $D_s$ and $D_t$ is called a {\em Hermitian symplectic monogenic} (or $h$-symplectic monogenic for short). 
\end{definition}
\par 
These two symplectic Dirac operators $(D_s,D_t)$ will serve as the analogue of the operators $(\up_x,\up_x^{\mathbb{J}})$. One can develop the analogy with the classical situation even further by introducing the symplectic analogues of the operators $\up_z$ and $\up_z^{\dagger}$, see \eqref{dz}.
Recall that these Hermitian Dirac operators $\up_z$ and $\up_z^{\dagger}$ were introduced using the Witt basis vectors. In the symplectic setting, we will do something similar by introducing the operators $D_z$ and $D_z^{\dagger}$, known in the literature as the \textit{symplectic Dolbeault operators} \cite{dol}. These are defined using the lowering operator $L_j=q_j+\partial_{q_j}$ and raising operator $R_j=q_j-\partial_{q_j}$ $(1\leq j\leq n)$ acting on symplectic spinors. 
\begin{definition}
	The symplectic Dolbeault operators are defined as follows: 
	\begin{align*}
	D_z :=\frac{1}{2}(D_s+iD_t)&=-\sum_{j=1}^nL_j\partial_{z_j}= -\langle \underline{L},\up_z\rangle \\
	D_z^{\dagger} :=\frac{1}{2}(D_s-iD_t)&=\sum_{j=1}^nR_j\partial_{{\overline{z}}_j}=\langle \underline{R},\up_{\overline{z}}\rangle \\
	X_z :=\frac{1}{2}(X_s+iX_t)&=\frac{i}{2}\sum_{j=1}^nL_j{\overline{z}_j}=\frac{i}{2}\langle \underline{L},\underline{\overline{z}}\rangle\\
	X_z^{\dagger} :=\frac{1}{2}(X_s-iX_t)&=\frac{i}{2}\sum_{j=1}^nR_j{{z}_j}=\frac{i}{2}\langle \underline{R},\uz\rangle\ ,
	\end{align*}
	where we used the short-hand notation notation $\langle\cdot,\cdot\rangle$ for the inner product.
\end{definition}
\par
The general philosophy in passing from the orthogonal setting to the symplectic one often amounts to replacing an anti-commutator $\{\cdot,\cdot\}$ by a commutator $[\cdot,\cdot]$. As a simple example we mention the definition for the Clifford algebra versus the Weyl algebra. A similar thing happens for the factorisation of the Laplace operator. Whereas the formula $\Delta =4 \{\up_z,\up_z^\dagger\}$ expresses $\Delta$ as an anti-commutator, a simple computation leads to the commutator relation $[D_s,D_t]=-i\Delta$ in the symplectic setting. An immediate consequence is the following: 
\begin{lemma}
	We have the inclusion $\ker(D_s) \cap \ker(D_t)\subset \ker(\Delta)$.
\end{lemma}
\noindent
This shows that also Hermitian symplectic Clifford analysis is a \textit{refinement} of harmonic analysis for the Laplace operator $\Delta$ on $\R^{2n}$. The dual operators $X_s$ and $X_t$ obviously satisfy a similar relation $
[X_s,X_t]=-ir^2$, 
where $r^2$ is the multiplication operator $r^2 = \sum_j(x_j^2 + y_j^2) = \sum_j|z_j|^2$. Putting everything together, this leads to a third realisation for the Lie algebra $\mathfrak{sl}(2)$ in our framework, next to the algebras $\operatorname{Alg}(D_s,X_s)$ and
$\operatorname{Alg}(D_t,X_t)$. The former was proved in \cite{DBHS} and used to obtain the symplectic Fischer decomposition, the latter is based on similar calculations. Indeed, it is well-known that $\operatorname{Alg}(\Delta,r^2) \cong \mathfrak{sl}(2)$, after a suitable normalisation of these operators. To summarise, we then have the following: 
$$\begin{tabular}{|c||c|c|}
\hline
\bfseries{Complex data}  & \bfseries{Orthogonal setting} & \bfseries{Symplectic setting} \\
\hline	\hline
none & $-\Delta=\up_x^2$ & $\Delta\neq D_s^2$ \\
\hline
complex structure $\mathbb{J}$ & \makecell{Real viewpoint\\$\Delta=\up_x^2=\mathbb{J}[\up_x]^2$
	\\Complex viewpoint
	\\$\Delta=4\{\up_z,\up_z^{\dagger}\}$} & \makecell{Real viewpoint\\$\Delta=[D_s,D_t]$
	\\Complex viewpoint
	\\$\Delta=[D_z,D_z^{\dagger}]$} \\
\hline
\end{tabular}$$
We will now describe a wide class of examples of $h$-symplectic monogenics, by making the link with holomorphic functions in several variables.
Let $f(\uz): \mathbb{C}^n\to \mathbb{C}$ be a complex-valued function in several complex variables which is of the class $\mathcal{C}^1(\Omega)$ (i.e.\ continuously differentiable). We say that $f$ is \textit{holomorphic in several variables} if $\partial_{\overline{z}_j}f(\uz)=0$ for all $1\leq j \leq n$, where $\partial_{\overline{z}_j}:=\frac{1}{2}(\partial_{x_j}+i\partial_{y_j})$ is the Cauchy-Riemann operator in the relevant variable. Note that we will often write functions $f(\uz)$ in terms of real variables, i.e. as $f(\ux,\uy)$. We can then prove the following: 
\begin{theorem}\label{habermannsolutions} Let $f(\ux,\uy)$ be a holomorphic function in several complex variables. Then the functions of the form  
	\begin{align*}
	F(\ux,\uy;\underline{q}) := f(\ux,\uy) e^{-\frac{1}{2}|\underline{q}|}
	\end{align*}
	are solutions of $D_s$ and $D_t$, i.e. symplectic $h$-monogenics.
\end{theorem}
\begin{proof}  
	Letting the symplectic Dirac operator $D_s$ act on $F(\ux,\uy;\underline{q})$ gives:
	\begin{align*}
	D_s\left(e^{-\frac{1}{2}|\underline{q}|^2}f(\ux,\uy)\right)
	&= e^{-\frac{1}{2}|\underline{q}|^2}\sum_k q_k(\partial_{x_k}+i\partial_{y_k})f(\ux,\uy)
	\end{align*}
	This is indeed equal to zero if $f(\uz) = f(\ux,\uy)$ is holomorphic in several variables. A similar calculation for $D_t$ completes the proof.
\end{proof}
\begin{remark}
	We note that Habermann \cite{MR2252919} studied the following system for the symplectic Dirac operators $D_s$ and $D_t$ on $\R^2$:
	\begin{align*}
	\begin{cases}
	(iq\partial_y-\partial_q\partial_x)f(x,y;q)&=0\\
	(iq\partial_x+\partial_q\partial_y)f(x,y;q)&=0
	\end{cases}
	\end{align*}
	and concluded that the solution are of the form $f=e^{-\frac{1}{2}q^2}H$ where $H$ is a holomorphic function. Proposition \ref{habermannsolutions} thus generalises these solutions to $\R^{2n}$. However, it turns out that there are much \textit{more solutions} when the dimension increases from 2 to $2n$ (with $n > 1$). If $n=3$ for instance, it is easily seen that 
	\begin{align*}
	F(\ux,\uy;\underline{q}) = e^{-\frac{1}{2} \left(q_1^2+q_2^2+q_3^2\right)} \left(2 \left(2 q_3^2-1\right) \left(x_2-i y_2\right)-4 q_2 q_3 \left(x_3-i y_3\right)\right)
	\end{align*}
	is indeed symplectic $h$-monogenic, although not of the holomorphic type. In what follows we will therefore describe the elements in $\ker(D_s) \cap \ker(D_t)$ more systematically. 
\end{remark}


\section{The symplectic Hermitian Howe duality}
To determine the Howe dual pair $\mathsf{U}(n) \times \mathfrak{g}$ which can be seen as the symplectic version of the dual pair $\mathsf{U}(n) \times \mathfrak{osp}(2|2)$ from Theorem \ref{hermitianhowe}, we start by describing the Lie algebra generated by the operators $D_s$ and $D_t$ and their (Fischer) duals $X_s$ and $X_t$. Note that we have already obtained the following three realisations for the Lie algebra $\mathfrak{sl}(2)$:
\begin{align*}
\operatorname{Alg}\{D_s,X_s,\mathbb{E}+n\} \cong \operatorname{Alg}\{D_t,X_t,\mathbb{E}+n\} \cong \Alg(\Delta,r^2,\mathbb{E}+n)\cong \mathfrak{sl}(2)\ .
\end{align*}
The three copies of $\mathfrak{sl}(2)$ do \textit{not} commute mutually however, because for instance
\begin{align*}[D_t,X_s] &= \sum_{j=1}^n \big( i(x_j\partial_{y_j}
-y_j\partial_{x_j})+\partial_{q_j}^2-q_j^2\big) \neq 0\ .
\end{align*}
The operator above will thus be required to close the Lie algebra, so let us give it a name: 
\begin{definition}
	The operator $\mathcal{O}$, acting on $\mathbb{S}^\infty$-valued functions $f(\ux,\uy;\underline{q})$ on $\R^{2n}$, is defined by means of 
	\begin{align*}
	\mathcal{O} := \sum_{j=1}^n \big( i(x_j\partial_{y_j}
	-y_j\partial_{x_j})+\partial_{q_j}^2-q_j^2\big)\ .
	\end{align*}
\end{definition}
It turns out that by adding the operator $\mathcal{O}$ to the operators from above in our $\mathfrak{sl}(2)$-triplets, we will indeed get an algebra of dimension 8 which closes under the commutator bracket $[\cdot,\cdot]$. Let us first introduce this Lie algebra in terms of matrices:
\begin{definition}
	The Lie algebra $\mathfrak{su}(1,2)$ is a (quasi-split) real form of the complex Lie algebra $\mathfrak{sl}(3)$, and is defined in terms of matrices as follows
	\begin{align*}
	\mathfrak{su}(1,2) = \left\{\begin{pmatrix}
	\alpha	&	\beta	&	ic \\
	\gamma	&	\overline{\alpha} - \alpha &	-\overline{\beta}\\
	id		&	-\overline{\gamma}	& - \overline{\alpha}.
	\end{pmatrix} \in \C^{3 \times 3}: c,d\in \R \text{ and } \alpha,\beta,\gamma\in\C\right\}
	\end{align*}
\end{definition}
A basis for this Lie algebra is given by the the following set of matrices:
\begin{align*}
\begin{matrix}
H_1	=	\begin{pmatrix}
1	&	0	&	0	\\
0	&	0	&	0	\\
0	&	0	&	-1
\end{pmatrix} &  X_1	=	\begin{pmatrix}
0	&	1	&	0	\\
0	&	0	&	-1	\\
0	&	0	&	0
\end{pmatrix} 
& X_2	=	\begin{pmatrix}
0	&	i	&	0	\\
0	&	0	&	i	\\
0	&	0	&	0
\end{pmatrix}
& X_3	=	\begin{pmatrix}
0	&	0	&	i	\\
0	&	0	&	0	\\
0	&	0	&	0
\end{pmatrix}\\
H_2	=	\begin{pmatrix}
i	&	0	&	0	\\
0	&	-2i	&	0	\\
0	&	0	&	i
\end{pmatrix} &
Y_1	=	\begin{pmatrix}
0	&	0	&	0	\\
1	&	0	&	0	\\
0	&	-1	&	0
\end{pmatrix}
& Y_2	=	\begin{pmatrix}
0	&	0	&	0	\\
i	&	0	&	0	\\
0	&	i	&	0
\end{pmatrix}
& Y_3	=	\begin{pmatrix}
0	&	0	&	0	\\
0	&	0	&	0	\\
i	&	0	&	0
\end{pmatrix}
\end{matrix}
\end{align*}
The commutators $[M_1,M_2]$ between two matrices $M_1$ and $M_2$ in $\mathfrak{su}(1,2)$ give rise to the table below (with $M_1$ in a row and $M_2$ in a column). 
\begin{table}[h!]
	\begin{center}
		\begin{tabular}{c|c c c c c c c c }
			$[\cdot,\cdot]$ & $H_1$ & $H_2$ & $X_1$ & $X_2$ & $X_3$ & $Y_1$  & $Y_2$  & $Y_3$ \\
			\hline 
			$H_1$ & $0$ & $0$ &  $X_1$& $X_2$ & $2X_3$  & $-Y_1$  & $-Y_2$  & $-2Y_3$  \\
			
			$H_2$ & 0 & 0 & $3X_2$ & $-3X_1$ & $0$ &$-3Y_2$ & $3Y_1$ & 0 \\
			
			$X_1$ &$-X_1$  & $-3X_2$ &0  & $2X_3$ & $0$ & $H_1$ &$H_2$  & $-Y_2$ \\
			
			$X_2$ &  $-X_2$& $3X_1$ & $-2X_3$ &$0$  & $0$ & $H_2$ & $-H_1$  &$-Y_1$  \\
			
			$X_3$ & $-2X_3$ & $0$  & $0$ & $0$ & $0$ & $-X_2$ & $-X_1$ & $-H_1$ \\
			
			$Y_1$ & $Y_1$ & $3Y_2$ & $-H_1$ & $-H_2$ & $X_2$ & $0$ & $-2Y_3$ &$0$  \\
			
			$Y_2$ &$Y_2$  & $-3Y_1$  &$ -H_2$  & $H_1$ & $X_1$ & $2Y_3$ & $0$ & $0$ \\
			
			$Y_3$ & $2Y_3$ & $0$ & $Y_2$  & $Y_1$  & $H_1$   & $0$ & $0$ & $0$ \\
		\end{tabular}
		\label{su(1,2)tabel}
	\end{center}
\end{table}
\begin{theorem}\label{su_operatorform}
	The algebra generated by the symplectic Dirac operators $D_s,D_t$ and their adjoints $X_s,X_t$ gives rise to a copy of the Lie algebra $\mathfrak{su}(1,2)$. 
\end{theorem}
\begin{proof}
	It follows from straightforward calculations that the eight operators can be identified with the matrices from above as follows: 
	\begin{align*}
	(H_1,H_2,X_1,X_2,X_3,Y_1,Y_2,Y_3) \leftrightarrow (\mathbb{E}+n,i\mathcal{O},iX_t,iX_s,-\frac{i}{2}r^2,-D_t,D_s,\frac{i}{2}\Delta)\ .
	\end{align*}
	Together with the table above, this proves the theorem.
\end{proof}


\section{Decomposition of the symplectic spinor space}
In Section 2 we studied the decomposition of the (standard) spinor space $\mathbb{S}$ in terms of the spin-Euler operator $\beta$. In this section, we will do the same thing for the infinite-dimensional symplectic spinor space $\mS^{\infty}$. For this we will need the powers $\rho^k$ of the determinant representation (where $k \in \R$), given by 
\begin{align*}
\rho^k:\mathfrak{u}(n)\to \C : M\mapsto \det(M)^k\ ,    
\end{align*}
with highest weight $(k,k,\dots,k)_{\mathfrak{u}(n)}$. Note that these weights correspond to the action of a Cartan algebra for the complexification of $\mathfrak{u}(n)$, where the elements $H_j := iB_{jj}$ are chosen (see Lemma \ref{u-realisation}). As for $\C$-valued polynomials, we first mention the following well-known result: 
\begin{theorem}\label{hwpoly}
	The space $\mathcal{H}_{a,b}(\mathbb{R}^{2n},\C)$ of bi-homogeneous harmonics of degree $(a,b)$ is an irreducible $\mathfrak{u}(n)$-modules of highest weight $(a,0,\dots,0,-b)_{\mathfrak{u}(n)}$ and has highest weight vector $w_{a,b}(\uz,\overline{\uz})=z_1^a\overline{z}_n^b$. Moreover, the dimension of the module is given by \[ \dim(\mathcal{H}_{a,b}(\mathbb{R}^{2n},\C)) =\frac{n + a + b - 1}{n-1}{a + n - 2 \choose n - 2 }{b + n - 2 \choose n - 2 }. \]
\end{theorem}
Next, we consider the branching of the infinite-dimensional symplectic spinors, when considered as a $\mathfrak{u}(n)$-module. 
\begin{theorem}\label{Srinf}
	The symplectic spinor space $\mathbb{S}^{\infty}$ decomposes as follows:
	\begin{align*}
	\mathbb{S}^{\infty}\bigg\downarrow^{\mathfrak{sp}(2n)}_{\mathfrak{u}(n)} =\bigoplus_{k=0}^{\infty}\mathbb{S}^{\infty}_{(r)} = \bigoplus_{r=0}^{\infty}\left(-\frac{1}{2},\dots,-\frac{1}{2}-r
	\right)_{\mathfrak{u}(n)}\ . 
	\end{align*}
	The space $\mathbb{S}^{\infty}_{(r)}$ hereby has dimension ${n+r-1 \choose r}$.
\end{theorem}
\begin{proof}
	The branching of the metaplectic representation $\mathbb{S}^\infty = \mathbb{S}^\infty_+ \oplus \mathbb{S}^\infty_-$ to the maximal compact subalgebra $\mathfrak{u}(n) \subset \mathfrak{sp}(2n)$ was treated in \cite{King2000}, and is given by
	\begin{align*}
	\mathbb{S}_+^{\infty}&\to \left(-\frac{1}{2},\dots,-\frac{1}{2}\right)_{\mathfrak{u}(n)}\otimes \bigoplus_{r\in 2\mathbb{N}}(0,0,\dots,-r)_{\mathfrak{u}(n)}
	\\		\mathbb{S}_-^{\infty}&\to \left(-\frac{1}{2},\dots,-\frac{1}{2}\right)_{\mathfrak{u}(n)}\otimes \bigoplus_{r\in 2\mathbb{N}+1}(0,0,\dots,-r)_{\mathfrak{u}(n)}
	\end{align*}
	Working out these tensor products leads to the desired result. The dimension of these spaces can be calculated using the Weyl dimension formula:
	\begin{align*}
	\dim \mathbb{S}^{\infty}_{(r)} = \prod_{j=2}^n \frac{j+k-1}{j-1} = {n+r-1 \choose r}\ ,
	\end{align*}
	which concludes the proof. 
\end{proof}
\begin{remark}
	Note that this dimension coincides with the dimension of the space $\mathcal{P}_r(\R^{2n},\C)$ of $r$-homogeneous polynomials. This is not a coincidence, since the spaces $\mS^{\infty}_{(r)}$ are known as the eigenspaces for the bosonic quantum oscillator. 
\end{remark}
In the following scheme we summarise the comparison between classical spinors in $\mS$ and symplectic spinors in $\mS^{\infty}$:
\begin{table}[h!]
	\begin{center}
		\begin{tabular}{|c|c |c|}
			\hline
			\textbf{Orthogonal} & \textbf{Symplectic} \\
			\hline\hline
			$\mS = \bigoplus_{r=0}^n\mS_{(r)}$  & $\mS^{\infty} = \bigoplus_{r=0}^{\infty}\mS_{(r)}^{\infty}$ \\
			\hline
			$\dim\mathbb{S}_{(r)}=\dim \bigwedge^r V={\binom {n}{r}}$  & $\dim\mathbb{S}^{\infty}_{(r)}=\dim\operatorname{Sym}^r(V)={n+r-1\choose r}$  \\
			\hline
			spin-Euler $\beta$& Hermite operator $\Delta_q - |\underline{q}|^2$
			\\
			\hline 
		\end{tabular}
	\end{center}
\end{table}

\begin{lemma}\label{Srweight}
	The unitary representation $\mathbb{S}_{(r)}^{\infty}$ is generated by the weight vector 
	\begin{align*}
	w(\underline{q}) = e^{-\frac{1}{2}|\underline{q}|^2}\mathsf{H}_r(q_n)\ ,    
	\end{align*}
	where $\mathsf{H}_r(q_n)$ is a Hermite polynomial of degree $r$ in the variable $q_n$.
\end{lemma}
\begin{proof}
	For all indices $1 \leq j < n$ we directly obtain that
	\begin{align*}
	H_j w(\underline{q}) = \frac{1}{2}\left(\partial_{q_j}^2 - q_j^2\right)w(\underline{q}) = -\frac{1}{2}w(\underline{q})\ .
	\end{align*}
	The action of the last Cartan element $H_n$ is slightly more complicated: 
	\begin{align*}
	\frac{1}{2}\left(\partial_{q_n}^2 - q_n^2\right)w(\underline{q}) =  \frac{1}{2}  e^{-\frac{1}{2}|\underline{q}|^2} \left(4 (r-1) r \mathsf{H}_{r-2}\left(q_n\right)-\mathsf{H}_r\left(q_n\right)-4 q_n r \mathsf{H}_{r-1}\left(q_n\right)\right)\ .	
	\end{align*}
	Using the recursion relation for Hermite polynomials, which says that 
	\begin{align*}
	\mathsf{H}_{r}(q_n) =2q_n\mathsf{H}_{r-1}(q_n)-2(r-1)\mathsf{H}_{r-2}(q_n)\ ,
	\end{align*}
	we can rewrite this action as
	\[  \frac{1}{2} \left(\partial_{q_n}^2 - q_n^2\right)w(\underline{q}) =-\frac{1}{2}e^{-\frac{1}{2}|\underline{q}|^2}(2r+1)\mathsf{H}_r(q_n) =-\frac{1}{2}\left(r+\frac{1}{2}\right)w(\underline{q})\ .\]
	One further verifies that $w(\underline{q})$ is a solution for the positive root vectors, which finishes the proof.
\end{proof}


\section{The symplectic Hermitian Fischer decomposition}
Now that we have obtained the reduction of the symplectic spinors $\mathbb{S}^{\infty}$ in terms of the spaces  $\mathbb{S}^{\infty}_{(r)}$ and have identified our candidate for the Howe dual pair as $\mathsf{U}(n)\times \mathfrak{su}(1,2)$, we are ready to formulate the Fischer decomposition for the space $\mathcal{H}_{a,b}(\R^{2n},\C)\otimes \mathbb{S}_{(r)}^{\infty}$. In terms of $\mathfrak{u}(n)$-irreducible representations, we need to compute the tensor product 
\begin{align} 
(a,0,\dots,0,-b)_{\mathfrak{u}(n)} \otimes \left(-\frac{1}{2},\dots,-\frac{2r+1}{2}\right)_{\mathfrak{u}(n)}\ .
\end{align}\label{tensorr}
Our strategy is the following:
\begin{enumerate}
	\item First we will compute $(a,0,\dots,0,-b)_{\mathfrak{u}(n)}\otimes (0,\cdots,0,-r)_{\mathfrak{u}(n)}$ using a suitable `dualisation argument' and then determine the tensor product of the resulting summands with the determinant representation $\rho^{-\frac{1}{2}}$.
	\item Of course, the procedure in step (i) yield a lot of information, but it only gives the decomposition of the tensor product \textit{up to an isomorphism}. In order to make the isomorphism into an equality (as was in Theorem \ref{HFD}) we need the so-called \textit{embedding factors} which will turn out to be unitary invariant differential operators.  
\end{enumerate}
For two irreducible representations for $\mathfrak{u}(n)$ with highest weights $\alpha$ and $\beta$, the tensor product decomposes as $\alpha \otimes \beta \cong \bigoplus_{\gamma} \mu_{\gamma}\gamma$, where $\mu_{\gamma}$ is the multiplicity. As a particular case, we mention the following result (from \cite{zhelobenko1973compact}): 
\begin{align*}
(m,0,\dots,0)_{\mathfrak{u}(n)}\otimes (\beta_1,\dots,\beta_n)_{\mathfrak{u}(n)}\cong \bigoplus\:(\beta_1+\nu_1,\dots,\beta_n+\nu_n)_{\mathfrak{u}(n)}    
\end{align*}
where the summation is such that $\nu_1+\cdots+\nu_n=m$. Moreover, we have the conditions $0 \leq \nu_1 \leq m$ and $0\leq \nu_{k}\leq \beta_{k-1} - \beta_{k}$ for $2 \leq k \leq n$. Invoking the fact that $(V \otimes W)^\ast \cong V^\ast \otimes W^\ast$ (the dualisation mentioned above), it then suffices to put 
\begin{align*}
m = r \quad \mbox{and} \quad (\beta_1,\ldots,\beta_n) = (b,0,\ldots,0,-a)
\end{align*}
in the formula above, after which one can take the dual (or contragradient) and a tensor product with the (one-dimensional) determinant representation $\rho^{-\frac{1}{2}}$ to arrive at the theorem below. Note that from now on, we restrict ourselves to the case $n \geq 3$. 
\begin{theorem}\label{tensorP}
	We have the following tensor product decomposition:
	\[ \mathcal{H}_{a,b}\otimes \mathbb{S}_{(r)}^{\infty} \cong  \bigoplus_{j=0}^a\bigoplus_{i=0}^b \left(a-j-\frac{1}{2},-\frac{1}{2},-\frac{1}{2},\cdots,-i-\frac{1}{2},i+j-b-r-\frac{1}{2}\right)_{\mathfrak{u}(n)}\ ,\]
	whereby the indices $i$ and $j$ satisfy the additional constraint $i + j \leq r$. 
\end{theorem}
\begin{remark}
	If $a, b \geq r$, the direct sum at the right-hand side has a maximal number of summands, equal to the triangular number $t_{r+1} = \frac{1}{2}(r+1)(r+2)$. This should be contrasted with the classical case (the case $\mcH_{a,b} \otimes \mathbb{S}_{(r)}$ in Hermitean Clifford analysis), where the maximal number of summands is always bounded by $4$. 
\end{remark}
As usual in the framework of Fischer decompositions, we will now try to make the information in the theorem above more explicit. 
\begin{remark}\label{H123}
	Let us for instance calculate $\mathcal{H}_{1,2}(\mathbb{C}^3,\mathbb{C})\otimes \mathbb{S}_{(3)}^{\infty}$. Using the formula from above we have:
	\begin{align*}
	(1,0,-2)_{\mathfrak{u}(3)}\otimes (0,0,-3)_{\mathfrak{u}(3)}' &\cong (1,0,-5)_{\mathfrak{u}(3)}' \oplus (1,-1,-4)_{\mathfrak{u}(3)}' \oplus (1,-2,-3)_{\mathfrak{u}(3)}'\\
	& \oplus (0,0,-4)_{\mathfrak{u}(3)}'\oplus (0,-1,-3)_{\mathfrak{u}(3)}'\oplus (0,-2,-2)_{\mathfrak{u}(3)}',
	\end{align*}
	where we have introduced the shorthand notation $(a,b,c)':=\left(a-\frac{1}{2},b-\frac{1}{2},c-\frac{1}{2}\right)$.
\end{remark}
Before we can express the summands in the decomposition for $\mathbb{S}^{(r)}$-valued harmonics in terms of $h$-monogenics, we introduce the following: 
\begin{definition}
	For positive integers $a, b$ and $r$ with $b \leq r$ we use the notation $\mathcal{M}_{a,b}^{(r)}$ for the irreducible representation highest weight 
	\begin{align*}
	\left(a-\frac{1}{2},-\frac{1}{2},\ldots,-b-\frac{1}{2},-r-\frac{1}{2}\right)\ . 
	\end{align*}
	We restrict ourselves to the case $n \geq 3$ here. 
\end{definition}
The dimension of these spaces follows from the Weyl dimension formula, for which we refer to a standard textbook: 
\begin{lemma}\label{dimensionformulaBarnes}
	Let $a,b, r \geq 0$ be arbitrary integers with $b \leq r$, then the dimension of the $\mathfrak{u}(n)$-irreducible representation $\mathcal{M}_{a,b}^{(r)}$ is given by
	\begin{align*}
	\dim \mathcal{M}_{a,b}^{(r)}=\frac{\Gamma(a+n-2)\Gamma(b+n-2)\Gamma(r+n-1)}{\Gamma(a+1)\Gamma(b+1)\Gamma(r+2)}\frac{(a+b+n-2)(a+r+n-1)(r-b+1)}{(n-1)!(n-2)!(n-3)!}\ .
	\end{align*}
\end{lemma}
Although this is not absolutely necessary, one can now use the formula above to verify Theorem \ref{tensorP} by a dimensional argument. Indeed, using symbolic computation software (the authors used Mathematica), one can verify that the following numerical equality holds: 
\begin{lemma}
	For arbitrary positive integers $a, b, r \geq 0$ for which $b \leq r$, we have that
	\begin{align*}
	\sum_{i=0}^b\sum_{j=0}^a\dim\mathcal{M}_{a,b}^{(r)}(\R^{2n}) &= {r+n-1 \choose n-1}\frac{(a+n+n-1)\Gamma(a+n-1)\Gamma(b+n-1)}{(n-1)!(n-2)!\Gamma(a+1)\Gamma(b+1)}\ ,
	\end{align*}
	whereby the indices $i$ and $j$ satisfy the additional constraint $i + j \leq r$. 
\end{lemma}
Our goal for the rest of this section is to show that these spaces are precisely the spaces of (symplectic) $h$-monogenics of a degree $(a,b;r)$, so that we can also define these spaces in an equivalent way as follows:
\begin{align}\label{M_spaces_kernel}
\mathcal{M}_{a,b}^{(r)}(\R^{2n}) = \mathcal{P}_{a,b}(\R^{2n},\mathbb{S}^\infty_{(r)}) \cap \big(\ker(D_s) \cap \ker(D_t)\big)\ .
\end{align}
\begin{remark}\label{br_condition}
	Note that the condition $b \leq r$ on the parameters may seem a bit strange, since this implies that spaces of $h$-monogenics will not always exist. However, a similar phenomenon also occurs in the classical case. One can for instance see that there are no $h$-monogenic polynomials of degree $(a,0)$ taking values in the homogeneous spinor space $\mathbb{S}_{(0)}$. So even in Hermitean Clifford analysis, the existence of the spaces $\mcM_{a,b}^{(j)}$ is governed by restrictions on the parameters $(a,b;j)$. 
\end{remark}
The inclusion from left to right in formula (\ref{M_spaces_kernel}) is rather immediate, the opposite direction requires more work. To see why
\begin{align}
\mathcal{M}_{a,b}^{(r)} \subset \mathcal{P}_{a,b}(\R^{2n},\mathbb{S}_{(r)}^\infty) \cap \big(\ker(D_s) \cap \ker(D_t)\big)\ ,   
\end{align}
it suffices to note that we have an explicit model for the highest weight vector which generates the representation under the action of the negative root vectors in $\mathfrak{u}(n)$. Indeed, recalling the definition for the (raising) operators $R_j=q_j-\partial_{q_j}$ and the Hermite polynomials $\mathsf{H}_k(t)$, it is not hard to see that the function 
\begin{align*}
w_{a,b;r}(\ux,\uy;\underline{q}) &:= z_1^a\big(\overline{z}_{n-1}R_n-\overline{z}_nR_{n-1}\big)^b \mathsf{H}_{r-b}(q_n)e^{-\frac{1}{2}|\underline{q}|^2}\ ,    
\end{align*}
satisfies the necessary requirements. It suffices to consider the action of the positive root operators 
\begin{align*}
C_{jk} + iA_{jk} = 2\big(z_j\partial_{z_k} -\overline{z}_k\partial_{\overline{z}_j}) +(q_j+\partial_{q_j})(\partial_{q_k}-q_k)\ ,
\end{align*}
where $1 \leq j < k \leq n$, and the Cartan elements 
\begin{align*}
H_j = H_{j}=-i(z_j\partial_{z_j}-\overline{z}_j\partial_{{\overline{z}}_j})+\frac{i}{2}\left(q_j^2-\partial_{q_j}^2\right)
\end{align*}
for $1 \leq j \leq n$. The former all act trivially on $w_{a,b;r}(\ux,\uy;\underline{q})$, whereas the latter indeed give the desired eigenvalues. Next, we prove that this function is also a solution the symplectic Dirac operators. One has that 
\begin{align*}
D_z w_{a,b;r}(\ux,\uy;\underline{q}) &= az_1^{a-1}L_1 (\overline{z}_{n-1}R_n-\overline{z}_nR_{n-1})^b \mathsf{H}_{r-b}(q_n)e^{-\frac{1}{2}|\underline{q}|} = 0\ , 
\end{align*}
since the lowering operator $L_1$ acts on the Gaussian only (here we have used the fact that $n \geq 3$, so that we can commute $L_1$ with the operators $R_n$ and $R_{n-1}$). A similar calculation for the operator $D_z^\dagger$ then proves that $w_{a,b;r}(\uz,\overline{\uz})$ is indeed $h$-monogenic.\\
To prove the opposite inclusion in equation (\ref{M_spaces_kernel}), we will argue as follows: given the fact that symplectic $h$-monogenics are always harmonic, we certainly know that 
\begin{align*}
\mathcal{P}_{a,b}(\R^{2n},\mathbb{S}^\infty_{(r)}) \cap \big(\ker(D_s) \cap \ker(D_t)\big) \subset \mcH_{a,b}(\R^{2n},\C) \otimes \mathbb{S}^\infty_{(r)}\ ,
\end{align*}
whereby we assume that $b \leq r$. Now, we know how the right-hand side decomposes (see Theorem \ref{tensorP}), so any polynomial $P(\ux,\uy)$ of degree $(a,b)$ on $\R^{2n}$ taking values in $\mathbb{S}^{(r)}$ which satisfies $D_s P = D_t P = 0$ can be written as a linear combination of elements belonging to the spaces predicted by this theorem. However, in what follows we will prove that there is a {\em unique} summand in the decomposition containing $h$-monogenics, and that all the other summands are orthogonal (in a suitable sense) to the space of $h$-monogenics. This will thus force $P(\ux,\uy)$ to be an element of that unique component, which will turn out to be the one generated by $w_{a,b;r}(\ux,\uy;\underline{q})$. \\
\par
The main idea behind the proof announced in the previous paragraph involves the construction of the so-called `embedding factors'. These are (in full generality) the $\mathfrak{u}(n)$-invariant linear maps injecting $\gamma$ into $\alpha \otimes \beta \cong \bigoplus \mu_\gamma \gamma$. For example, the vector variable $\uz$ appearing in Theorem \ref{HFD} can be seen as an example of such an embedding map, sending $h$-monogenics of certain degree into the tensor product $\mcH_{a,b} \otimes \mathbb{S}_{(r)}$. In the setting of Howe dual pairs, these embedding factors are typically given by the operators dual (in the sense of the Fischer inner product) to the differential operators defining the (monogenic) solution spaces. This is also the case in the present situation, since one can show the following: 
\begin{lemma}
	Let $0 \leq \ell \leq b$ and $0\leq k\leq b$, then the following properties hold:
	\begin{align*}
	X_z^{\ell}\mathcal{M}_{a,b-\ell}^{(r+\ell)} \in \mathcal{H}_{a,b}\otimes \mathbb{S}^{\infty}_{(r)} \quad \mbox{and} \quad (X_{z}^\dagger)^{k}\mathcal{M}_{a-k,b}^{(r-k)} \in \mathcal{H}_{a,b}\otimes \mathbb{S}^{\infty}_{(r)}\ .
	\end{align*}
	Note that these statements only hold if the parameters characterising the $\mcM$-spaces make sense, so $b - \ell \leq r + \ell$ and $b \leq r - k$ respectively. 
\end{lemma}
\begin{proof}
	It is clear that the operator $X_z^\ell$ raises the bi-degree from $(a,b-\ell)$ to $(a,b)$, and that the spinor-degree is lowered from $r+\ell$ to $r$. It thus suffices to show that the result is indeed harmonic. To that end, we note that $[\Delta_z,X_z] = -D_z$ (where we have put $\Delta_z = \sum \partial_{z_j}\bar{\partial}_{z_j}$) and $[D_z,X_z] = 0$. A similar calculation for the operator $X_z^\dagger$ then proves the lemma. 
\end{proof}
It is now tempting to think that one has a handle on all embedding factors, through combinations of powers of $X_z$ and $X_z^\dagger$. Unfortunately, the situation is slightly more complicated, because a simple calculation shows that for instance the product operator $X_zX_z^\dagger$ cannot be used to embed $\mcM_{a,b}^{(r)}$ into the space of $\mathbb{S}^\infty_{(r)}$-valued harmonics. It turns out that one needs a slightly more complicated algebraic structure to characterise the embedding factors, expressed in terms of a so-called transvector algebra. Note that this (quadratic) algebra already appeared in the setting of Howe dualities in harmonic analysis for the Laplace operator, see for instance \cite{DBER}. The construction of this type of algebra hinges upon the existence of a reductive subalgebra inside a Lie algebra, hence the following result:  
\begin{lemma}
	The Lie subalgebra $\mathfrak{sl}(2)$ is reductive in $\mathfrak{su}(1,2)$. 
\end{lemma}
\begin{proof}
	We need to prove that $\mathfrak{su}(1,2)$ decomposes as $\mathfrak{sl}(2)\oplus \mathfrak{t}$, such that the subspace $\mathfrak{t}$ caries the adjoint action of $\mathfrak{sl}(2)$. If we now define $\mathfrak{t} := \textup{span}_\R\big(\mathcal{O},X_s,X_t,D_s,D_t\big)$ and $\mathfrak{sl}(2)$ as the subalgebra generated by the remaining operators $\Delta, r^2$ and $\mathbb{E} + n$, the result follows from instance from the isomorphism constructed in Theorem \ref{su_operatorform}. 
\end{proof}
\par
For the most general definition of a transvector algebra, or the (intimitely related) Mickelsson algebra, we refer to the literature: see for instance \cite{Mi,Zh}. In the present paper we will proceed by constructing this algebra through its generators, which is typically done in terms of the extremal projection operator. Denoting our transvector algebra by means of $\mathcal{Z}(\mathfrak{su}(1,2),\mathfrak{sl}(2))$, we note that its generators are constructed in terms of the extremal projector 
\begin{align*}
\pi_{\mathfrak{sl}(2)} &:= 1 + \sum_{j=1}^{\infty}\frac{(-1)^j}{j!} \frac{\Gamma(H+2)}{\Gamma(H+2+j)}Y^jX^j\ ,
\end{align*}
whereby the (suitably normalised) operators $X, Y$ and $H$ are respectively given by $X = \frac{1}{2}\Delta$, $Y = -\frac{1}{2}r^2$ and $H = -(\mathbb{E} + n)$. Note that there is no need to care about convergence issues for the formal series above, since we will never need more than a finite number of terms. Also the factors in the denominator are harmless, because the previous observation will always allow us to get rid of the denominators after multiplication with a suitable factor (technically speaking, this amounts to going from a transvector to a Mickelsson algebra). The upshot is that although the operators in the subalgebra $\mathfrak{t} \subset \mathfrak{su}(1,2)$ do not map harmonics to harmonics, one does have that $\pi_{\mathfrak{sl}(2)}(\mathfrak{t}) \subset \mbox{End}(\ker(\Delta))$. So the operators in $\pi_{\mathfrak{sl}(2)}(\mathfrak{t})$ and products thereof will be our embedding factors. Let us then prove the following technical result:
\begin{lemma}\label{xsxtcommute}
	\leavevmode 
	\begin{enumerate}
		\item The transvector projections of the operators $X_s$ and $X_t$, acting on harmonic functions, can be rescaled using the factor $H+2$, such that 
		\begin{align*}
		-(H+2)\pi_{\mathfrak{sl}(2)}X_s &= \left(YX - (H+2)\right)X_s = (\mathbb{E} + n - 2)X_s - \frac{1}{2}r^2 D_t\\
		-(H+2)\pi_{\mathfrak{sl}(2)}X_t &= \left(YX - (H+2)\right)X_t = (\mathbb{E} + n - 2)X_t + \frac{1}{2}r^2 D_s\ .
		\end{align*}
		In other words, only the first two terms of the extremal projection operator $\pi_{\mathfrak{sl}(2)}$ are required to do the projection. 
		\item These rescaled operators commute in $\operatorname{End}(\ker(\Delta))$.
	\end{enumerate}
\end{lemma}
\begin{proof}
	The first statement (i) follows from the fact that $[\Delta,[\Delta,X_s]]=0$ (and similarly for $X_t$). This thus means that the infinite sum in the extremal projector reduces to a finite sum of two terms only. The explicit expression then follows from straightforward calculations, hereby taking into account that everything is to be seen as an operator acting on $\ker(\Delta)$. For the second statement (ii) we will make use of the commutator equality
$
	[AB,CD] = A[B,C]D+[A,C]BD+CA[B,D]+C[A,D]B$
which allows us to say that
	\begin{align*}
	[\pi_{\mathfrak{sl}(2)}X_s,\pi_{\mathfrak{sl}(2)}X_t] &= \pi_{\mathfrak{sl}(2)}[X_s,\pi_{\mathfrak{sl}(2)}]X_t+ \pi_{\mathfrak{sl}(2)}^2[X_s,X_t]+\pi_{\mathfrak{sl}(2)}[\pi_{\mathfrak{sl}(2)},X_t]X_s\ .
	\end{align*}
	Note that we can commute any factor involving the Euler operator $\mathbb{E}$ through the product $r^2\Delta$, because it is homogeneous of degree zero, and that $[r^2,X_s] = 0$ such that 
	\begin{align*}
	[\pi_{\mathfrak{sl}(2)},X_s] &= -\frac{1}{4}\left[(\mathbb{E} + n - 2)^{-1}r^2\Delta,X_s\right] = -\frac{1}{4}r^2\left[\Delta(\mathbb{E} + n - 2)^{-1}r^2,X_s\right]\ .
	\end{align*}
	This is enough to see why $\pi_{\mathfrak{sl}(2)}[X_s,\pi_{\mathfrak{sl}(2)}] = 0$ (and similarly for the last term in the righ-hand side above). Indeed, we can hereby make use of the fact $\pi_{\mathfrak{sl}(2)}r^2=0$, which is one of the defining properties for the extremal projection operator. This also gets rid of the middle term, as $[X_s,X_t] \sim r^2$. Finally, it suffices to multiply the (trivial) commutator $[\pi_{\mathfrak{sl}(2)}X_s,\pi_{\mathfrak{sl}(2)}X_t]$ from the left with $(H+2)(H+1)$ to see that also the rescaled operators will commute. 
\end{proof}
One can obviously also consider the projections of $X_z$ and $X_z^\dagger$. Briefly denoting the rescaled projection operator by means of $\tilde{\pi}_{\mathfrak{sl}(2)} = -(H+2)\pi_{\mathfrak{sl}(2)}$, we get the following: 
\begin{definition}
	The operators $\mathbb{X}_z$ and $\mathbb{X}_z^{\dagger}$ are defined as follows:
	\begin{align*}
	\mathbb{X}_z &:= \tilde{\pi}_{\mathfrak{sl}(2)}X_z = (\mathbb{E} + n - 2)X_z + \frac{i}{2}|\uz|^2D_z\\
	\mathbb{X}_z^{\dagger} &:= \tilde{\pi}_{\mathfrak{sl}(2)}X_z^\dagger = (\mathbb{E} + n - 2)X_z^\dagger - \frac{i}{2}|\uz|^2D_z^\dagger\ .
	\end{align*}
\end{definition}
In view of Lemma \ref{xsxtcommute} it is clear that also $[\mathbb{X}_z,\mathbb{X}_z^{\dagger}]=0$. In particular, one can thus consider the polynomial algebra $\C[\mathbb{X}_z,\mathbb{X}_z^\dagger]$ in these operators, defining a subalgebra of the endomorphism algebra End$(\ker(\Delta))$. Note that in switching from $(X_s,X_t)$ to $(X_z,X_z^\dagger)$ we introduced a complex unit, which means that $\C[\mathbb{X}_z,\mathbb{X}_z^\dagger]$ should be seen as a subalgebra of the complexified transvector algebra $Z(\mathfrak{su}(1,2),\mathfrak{sl}(2)) \otimes \C$ here. Monomials in the algebra $\C[\mathbb{X}_z,\mathbb{X}_z^\dagger]$ are precisely the embedding operators we are after, hence the following definition: 
\begin{definition}
	For all positive integers $k, \ell \in \mathbb{N}$ we define the maps
	\begin{align*}
	\mathcal{I}_{k,\ell} := \mathbb{X}^k (\mathbb{X}^\dagger)^\ell \in \mbox{End}\big(\ker(\Delta)\big)\ ,  
	\end{align*}
	whereby $\mathcal{I}_{0,0} = \textup{Id}$ is the identity map. Each of these maps is invariant with respect to $\mathfrak{u}(n)$, in the sense that $\mathcal{I}_{k,\ell}$ commutes with all the generators. 
\end{definition}
It is clear that the map $\mathcal{I}_{k,\ell}$ raises the bi-degree of homogeneity in $(\uz,\uz^\dagger)$ from $(a,b)$ to $(a+\ell,b+k)$, and it maps spinors in $\mathbb{S}^\infty_{(r)}$ to elements in $\mathbb{S}^\infty_{(r - k + \ell)}$. As such, these maps can indeed be used as embedding maps: 
\begin{align*}
\mathcal{I}_{k,\ell} : \mathcal{M}_{a-\ell,b-k}^{(r+k-\ell)}(\R^{2n}) \longrightarrow \mcH_{a,b}(\R^{2n},\C) \otimes \mathbb{S}^{\infty}_{(r)}\ ,
\end{align*} 
where we momentarily ignore the conditions on the parameters (depending on the values of the parameters, the space at the left-hand side of the arrow may not exist). To finish the goal we have set right before Remark \ref{br_condition}, the following will come in handy: 
\begin{lemma}
	For all $(k,\ell) \neq (0,0)$, we have the following: 
	\begin{align*}
	\mathcal{I}_{k,\ell} : \ker(D_s) \cap \ker(D_t) \longrightarrow \big(\ker(D_s) \cap \ker(D_t)\big)^\perp\ ,
	\end{align*}
	whereby the orthogonality is expressed with respect to the Fischer inner product, defined on $\mathcal{P}(\R^{2n},\C) \otimes \mathbb{S}^\infty$ by means of
	\begin{align*}
	\langle P(\ux,\uy) \otimes \psi(\underline{q}) , Q(\ux,\uy) \otimes \phi(\underline{q}) \rangle := \big(P(\up_y,-\up_x)Q(\ux,\uy)\big)\bigg\vert_{x=y=0}\int_{\R^n}\overline{\psi}(\underline{q})\phi(\underline{q})d\underline{q}\ .
	\end{align*}
\end{lemma}
\begin{proof}
	Because the operator $\mathcal{I}_{k,\ell}$ is a monomial in the operators $\mathbb{X}_z$ and $\mathbb{X}_z^\dagger$ (note that we exclude the identity operator $\mathcal{I}_{0,0}$ here), it is sufficient to investigate $\mathcal{I}_{1,0}$ and $\mathcal{I}_{0,1}$ separately. These operators are projections (up to a multiplicative constant, coming from the Euler factor $H+n$) of suitable linear combinations of $X_s$ and $X_t$, so we can even restrict our attention to these operators instead. The result then follows from the fact that $X_s$ (resp. $X_t$) is the Fischer adjoint of $D_s$ (resp. $D_t$). 
\end{proof}
This allows us to prove the following: 
\begin{lemma}
	For all positive integers $b \leq r$, the space $\mathcal{M}_{a,b}^{(r)}(\R^{2n})$ can be defined as the space of $\mathbb{S}^\infty_{(r)}$-valued $h$-monogenic polynomials on $\R^{2n}$ of bi-degree $(a,b)$. 
\end{lemma}
\begin{proof}
	Since any $\mathbb{S}^{\infty}_{(r)}$-valued polynomial $P_{a,b}(\ux,\uy;\underline{q})$ of bi-degree $(a,b)$ in the kernel of the operators $D_s$ and $D_t$ is also harmonic, we know that it must belong to one of the spaces appearing at the right-hand side of Theorem \ref{tensorP}, embedded using one of the mappings $\mathcal{I}_{k,\ell}$ constructed above. However, we have just shown that the only map that preserves the property of being symplectic $h$-monogenic is $\mathcal{I}_{0,0}$, which means that $P_{a,b}(\ux,\uy;\underline{q}) \in \mathcal{M}_{a,b}^{(r)}(\R^{2n})$. 
\end{proof}
Putting everything together then leads to the following refinement of Theorem \ref{tensorP}: 
\begin{theorem}
	[Symplectic Hermitian Fischer decomposition] We have the following explicit decomposition for the space of spinor-valued harmonics:
	\begin{align*}\label{symplectic_HFD}
	\mathcal{H}_{a,b}(\mathbb{C}^n,\mathbb{C}) \otimes \mathbb{S}_{(r)}^{\infty} &= \bigoplus_{j = 0}^a \bigoplus_{i = 0}^b \mathcal{I}_{b-i,j}\mathcal{M}_{a-j,i}^{(b+r - (i+j))}(\R^{2n})\ ,
	\end{align*}
	where the indices are constrained by the condition $i + j \leq r$. This is a decomposition of $\mathbb{S}^\infty_{(r)}$-valued harmonics into symplectic $h$-monogenics. 
\end{theorem}
\begin{proof}
	The decomposition itself follows from \ref{tensorP}, together with the observation that the maps $\mathcal{I}_{k,\ell}$ belong to End$(\ker(\Delta))$ and therefore have the desired mapping properties. The last statement (i.e. the connection with symplectic $h$-monogenics) follows from the previous lemma. 
\end{proof}
Despite the obvious similarities between the theorem above and the classical Fischer decompositions (for the harmonic and monogenic case), there are also a few crucial differences. First of all, note that when the space $\mathcal{M}_{a,b}^{(r)}$ is included in the tensor product $\mcH_{a,b}(\R^{2n},\C) \otimes \mathbb{S}^\infty_{(r)}$, then this space is {\em not} given by the Cartan product 
\begin{align*}
(a,0,\ldots,0,-b) \boxtimes \mathbb{S}^\infty_{(r)} \cong  \left(a-\frac{1}{2},-\frac{1}{2},\ldots,-\frac{1}{2},-b-r-\frac{1}{2}\right)\ .
\end{align*}
This is clearly different from the aforementioned Fischer decompositions, in which the monogenics were always embedded as a Cartan product. Moreover, the vector space of ($r$-homogeneous) spinor-valued harmonic polynomials does not always contain symplectic $h$-monogenics. Put differently, for certain choices of indices $(a,b;r)$ one has that 
\begin{align*}
\mcH_{a,b}(\R^{2n},\C) \otimes \mathbb{S}^\infty_{(r)} \subset \big(\ker(D_s) \cap \ker(D_t)\big)^\perp\ .
\end{align*}
As a simple example, consider $\mcH_{0,1}(\R^{2n}) \otimes \mathbb{S}^\infty_{(0)} = \mathbb{X}_z \mathcal{M}_{0,0}^{(1)}$. This is reminiscent of the situation in (standard) Hermitean Clifford analysis, where one for instance has that $\mcH_{1,0}(\R^{2n}) \otimes \mathbb{S}_{(0)} = \uz \mcM_{0,0}^{(1)}$. 

\section{Conclusion and further research}\label{outlook}
In this paper we introduced a complex structure $\mathbb{J}$ on $\R^{2n}$ which gave rise to a symplectic version of Hermitian Clifford analysis, centered around the system of equations $D_sf = D_t f = 0$. In this new setting, the Lie algebra $\mathfrak{su}(1,2)$ then appeared naturally as a dual partner for the unitary group $\mathsf{U}(n)$. This leads to the first two rows in the overview below. In the classical case (middle column), the Hermitean system was further refined, leading to a system of 4 `quaternionic' invariant equations (governed by the group $\mathsf{Sp}(n) =\mathsf {Sp} (2n,\mathbb {C} )\cap \mathsf {U} (2n)$, see \cite{quater}). The symplectic analogue of this function theory, the empty slot in the overview, will be the topic of further investigation.
$$\begin{tabular}{c||c|c}
\hline
\bf Complex structures & \bf Orthogonal case &  \bf Symplectic case\\
\hline\hline
none & $\mathsf{SO}(n)\times \mathfrak{osp}(1|2)$ 
& \makecell{$\mathsf{Sp}(2n)\times \mathfrak{sl}(2)$\\ }
\\
\hline
K\"ahler $\mathbb{J}$ & $\mathsf{U}(n)\times \mathfrak{osp}(2|2)$
& $\mathsf{U}(n)\times \mathfrak{su}(1,2)$
\\
\hline 
Hyperk\"ahler $\mathbb{I,J,K}$ &  $\mathsf{Sp}(n)\times \mathfrak{osp}(4|2)$
&  ? \\
\hline
\end{tabular}$$
\subsection*{Acknowledgments}Guner Muarem was supported by the FWO-EoS project G0H4518N.
\nocite{*}
\bibliographystyle{abbrv}

\end{document}

%% file: instellingen.tex
\usepackage{xcolor}
\definecolor{chapnumberbg}{RGB}{0,0,0}
\definecolor{chapname}{RGB}{0,0,0}
\definecolor{kleur}{RGB}{0,0,0}
\usepackage{wrapfig}
\usepackage{ytableau}
\usepackage{tikz}
\usepackage{setspace}
\usepackage{amsthm}
\usepackage{xpatch}
\xpatchcmd\swappedhead{~}{.~}{}{}
\usepackage{graphicx}

\newcommand{\Z}{\mathbb{Z}}
\newcommand{\ux}{\underline{x}}
\newcommand{\uy}{\underline{y}}
\newcommand{\uz}{\underline{z}}

\usepackage{makecell}

\newcommand{\mS}{\mathbb{S}}
\usepackage{enumerate}

\usepackage{hyperref}

\newcommand{\mcM}{\mathcal{M}}
\newcommand{\mcH}{\mathcal{H}}

\newcommand{\up}{\underline{\partial}}
\newcommand{\mE}{\mathbb{E}}

\usepackage{tikz-cd}
\usepackage{theoremref}
\usepackage{epigraph}
\usepackage{dynkin-diagrams}
\newcommand{\Alg}{\operatorname{Alg}}

\usepackage{dynkin-diagrams}
\usepackage{framed}
\definecolor{shadecolor}{gray}{0.9}
\usepackage[framemethod=tikz]{mdframed}
\usepackage{tikz-cd}

\usepackage[framemethod=tikz]{mdframed}
  
\newcommand{\U}{\operatorname{\mathsf{U}}} 
\newcommand{\R}{\mathbb{R}}

\newcommand{\C}{\mathbb{C}}

\usepackage{xcolor}
\usepackage{amsmath,amssymb,amsthm}
\usepackage{nomencl}
\makenomenclature
\usepackage{enumerate}
\usepackage{textcomp,enumitem}
\newlist{mylist}{enumerate}{2}

\usepackage{epigraph}
\usepackage[toc]{glossaries}
\usepackage[font=small]{caption}
\usepackage{mathtools}

\usepackage{theoremref, enumitem}
\usepackage{hyperref}
\usepackage[all]{xy}
\usepackage{tikz-cd}
\usepackage{scrextend,mdwlist}
\renewcommand\labelenumi{\normalfont(\roman{enumi})}

\renewcommand\theenumi\labelenumi
\usepackage{pdfpages,enumerate}
\usepackage{epigraph}
\usepackage{fancybox}
\renewcommand{\phi}{\varphi}
\usepackage{url}

\usetikzlibrary{matrix}

\usetikzlibrary{calc}

\usepackage{xcolor}
\usetikzlibrary{matrix}
\usepackage{xcolor}

\usepackage{makeidx}

\usepackage[font=small  ]{caption}
\usepackage{xcolor}
\definecolor{ua}{rgb}{0,0,0}

\hypersetup{
	colorlinks=true,
	linktoc=none,
	allcolors={black},
	pdfborder={0 0 0},
}

\usepackage{amsmath, amsthm, amssymb, amsfonts, enumerate, dlfltxbcodetips}
\newtheorem*{theorem*}{Theorem}
\newtheorem*{stelling*}{Stelling}
\newtheoremstyle{mystyle}
{6pt}
{6pt}
{\normalfont}
{0pt}
{\scshape\bfseries}
{.}
{4pt}
{}
\theoremstyle{mystyle}
\newtheorem{definition}{Definition}[section]
\newtheorem*{definition*}{Definition}

\newtheorem{remark}[definition]{Remark}

\renewcommand{\subset}{\subseteq}
\newtheoremstyle{mystyle3}
{6pt}
{6pt}
{\sffamily}
{0pt}
{\scshape\bfseries}
{.}
{4pt}
{}
\theoremstyle{mystyle3}

\newtheoremstyle{mystyle2}
{6pt}
{6pt}
{\itshape}
{0pt}
{\scshape\bfseries}
{.}
{4pt}
{}
\theoremstyle{mystyle2}
\newtheorem{theorem}[definition]{Theorem}

\newtheorem{lemma}[definition]{Lemma}

\newtheorem{corollary}[definition]{Corollary}

\usepackage{hyperref}

\definecolor{headtitle}{RGB}{0,0,0} 
\colorlet{captionhead}{headtitle}

\def\Hrulefill{\leavevmode\leaders\hrule width 1ex height 0.7ex depth -0.6ex\hfill\kern0pt}

\usepackage{mdwlist}

\usepackage{enumitem}

\numberwithin{equation}{section}
\usepackage{amssymb}  
\usepackage{tikz}